\documentclass{amsart}

\usepackage{mathrsfs}
\usepackage{amssymb}
\usepackage{mathtools}
\usepackage{graphicx}

\newtheorem{thm}{Theorem}[section]
\newtheorem{lem}[thm]{Lemma}
\newtheorem{cor}[thm]{Corollary}
\newtheorem{prop}[thm]{Proposition}
\newtheorem{example}[thm]{Example}
\newtheorem{remark}[thm]{Remark}

\theoremstyle{definition}
\newtheorem*{ack}{Acknowledgment}

\newcommand*{\Z}{\mathbb{Z}}   
\newcommand*{\R}{\mathbb{R}}   
\newcommand*{\st}{\,:\,}   
\newcommand*{\card}[1]{\left\lvert #1 \right\rvert}   
\newcommand*{\conv}{\operatorname{Conv}}   
\newcommand*{\width}{\operatorname{w}}   
\newcommand*{\Vol}{\operatorname{Vol}}   
\newcommand*{\maps}{\colon}   
\newcommand*{\intr}[1]{{#1}^{\circ}}   
\newcommand*{\Rec}{{\rm Rec}}

\newcommand*{\Q}{\mathbb{Q}}

\newcommand*{\ci}{\conv^\circ}
\renewcommand*{\L}{\mathcal{L}}
\newcommand*{\D}{\mathcal{D}}
\newcommand*{\Pz}{P_\Z}
\newcommand*{\Pin}{P^\circ}
\newcommand*{\Piz}{P^\circ_\Z}
\newcommand*{\Pbz}{\partial P_\Z}
\newcommand*{\V}{\mathcal{V}}
\newcommand*{\F}{\mathcal{F}}
\newcommand*{\Fz}{F_\Z}

\newcommand*{\Fiz}{F^\circ_\Z}
\newcommand*{\GL}{{\rm GL}}

\newcommand*{\Gin}{G^\circ}
\newcommand*{\Giz}{G^\circ_\Z}

\newcommand*{\Hiz}{H^\circ_\Z}
\newcommand*{\Kz}{K_\Z}
\newcommand*{\Kbz}{\partial K_\Z}
\newcommand*{\Kin}{K^\circ}
\newcommand*{\Kiz}{K^\circ_\Z}
\newcommand*{\clK}{\overline{K}}
\newcommand*{\clKz}{\overline{K}_\Z}

\newcommand*{\clKiz}{\overline{K}^\circ_\Z}
\newcommand*{\deftobe}{\coloneqq}

\title[Lattice width directions and Minkowski's $3^d$-theorem]{Lattice width directions and\\Minkowski's $3^d$-theorem}

\author[J.~Draisma]{Jan Draisma}
\address[Jan Draisma]{
Department of Mathematics and Computer Science\\
Technische Universiteit Eindhoven\\
P.O. Box 513, 5600 MB Eindhoven, Netherlands\\
and Centrum voor Wiskunde en Informatica, Amsterdam, The
Netherlands}
\thanks{The first author is supported by DIAMANT, an NWO
mathematics cluster.}
\email{j.draisma@tue.nl}

\author{Tyrrell B. McAllister}
\address[Tyrrell B. McAllister]{
Max Planck Institute for Mathematics\\
Vivatsgasse 7, 53111 Bonn, Germany}
\thanks{}
\email{tmcal@mpim-bonn.mpg.de}

\author{Benjamin Nill}
\address[Benjamin Nill]{
Institut f\"ur Mathematik\\
Freie Universit\"at Berlin\\
Arnimallee 3, 14195 Berlin, Germany}
\thanks{}
\email{nill@math.fu-berlin.de}

\date{Draft}

\begin{document}

\maketitle

\begin{abstract}
   We show that the number of lattice directions in which a
   $d$-dimensional convex body in $\R^d$ has minimum width is at most
   $3^d-1$, with equality only for the regular cross-polytope.  This is
   deduced from a sharpened version of the $3^d$-theorem due to Hermann
   Minkowski (22 June 1864---12 January 1909), for which we provide two
   independent proofs.
\end{abstract}

\section{Introduction}

The \emph{lattice width} of a non-empty subset $S$ of $\R^d$ is a
well-studied invariant in the geometry of numbers.  It is defined to be
the infimum of $\sup(u(S))-\inf(u(S))$ as $u$ ranges over the set of non-zero
vectors in the lattice dual to $\Z^{d} \subset \R^{d}$ for which both
$\sup(u(S))$ and $\inf(u(S))$ are finite. In the case that this set is empty, the lattice width is defined to be $\infty$. 
If the lattice width is finite, the vectors attaining this infimum are called \emph{lattice width directions}.

\bigskip
\begin{figure}[ht]
\includegraphics[height=.15\textheight]{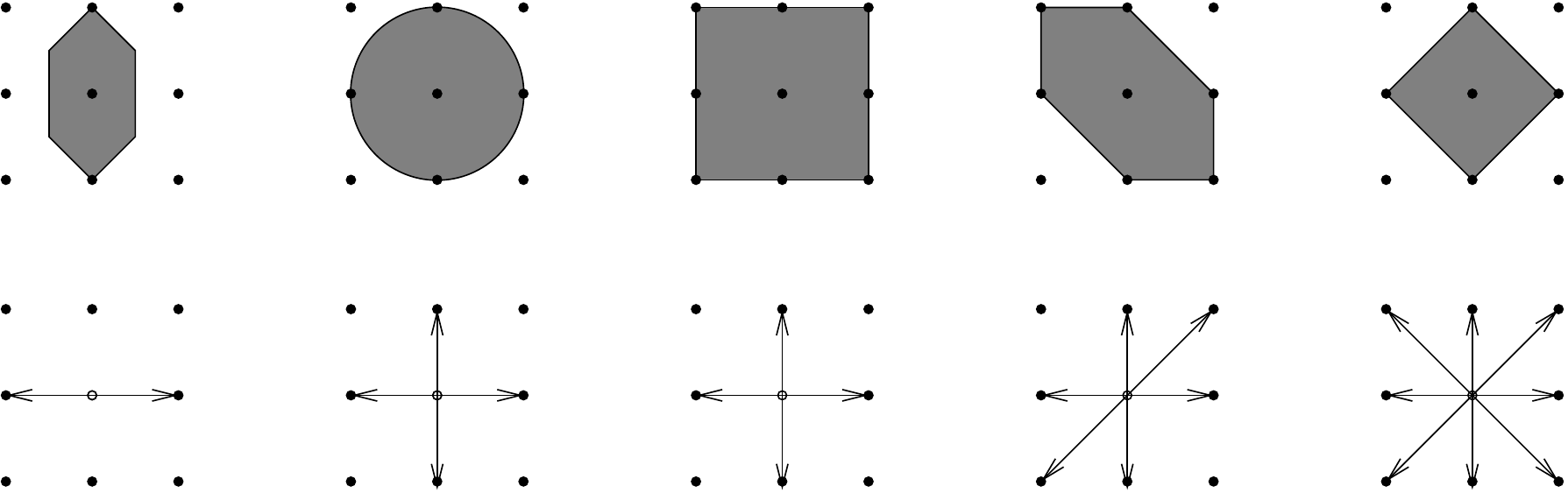}
\caption{Convex bodies in $\R^2$ and their lattice width directions}
\label{example}
\end{figure}
\smallskip

The lattice width and the set of lattice width directions is invariant under the action of matrices in $\GL_d(\Z)$ 
and under arbitrary translations of the convex body. The set of lattice width directions is also unchanged under scalings of the convex body. 
Note that in Figure \ref{example} the polygon on the right has many lattice width directions. 
This is an instance of a \emph{regular lattice cross-polytope}, which is defined as the convex
hull of $x \pm \lambda e_1, \ldots, x \pm \lambda e_d$ for some $x \in
\R^{d}$, $\lambda \in \R$, and a lattice basis $e_1, \dotsc, e_d$ of
$\Z^d$. Our main result shows that this is indeed the only extreme case.

\begin{thm}
   The number of lattice width directions of a non-empty subset $S$ of $\R^d$ 
   with $\dim(S) = d$ is at most $3^d-1$. Equality holds if and only the closure of the
   convex hull of $S$ is a regular lattice cross-polytope.
\label{main}
\end{thm}
   
We prove this result in Section 2. The proof relies on the following
strengthening of a theorem of Minkowski about centrally-symmetric
convex sets with only one interior lattice point. Denote by $\Kz$
the set of lattice points in the convex set $K$, and by $\Kiz$ the
set of lattice points in the relative interior of $K$.  While the most
well-known lattice point theorem by Minkowski gives an upper bound on
the volume of a centrally-symmetric convex set with only one interior
lattice point, the result we are interested in yields an upper bound on
the number of lattice points.  We say that $K$ is a \emph{standard lattice cube}
if there is a lattice basis $e_1, \dotsc, e_d$ of $\Z^d$ such that $K$
is the convex hull of $\pm e_1 \pm \dotsb \pm e_d$.

\begin{thm}
   Let $K \subseteq \R^d$ be a centrally-symmetric
   convex set.  If $\Kiz = \{0\}$, then $\card{\Kz} \leq 3^d$, with
   equality if and only if $K$ is a standard lattice cube.
\label{mink}
\end{thm}

We remark that there are centrally-symmetric compact convex 
sets $K$ with $\Kiz = \{0\}$ that are {\em not} contained in a standard lattice cube, 
see Remark 4.9 in \cite{Nill06b}.

\smallskip
The upper bound in Theorem~\ref{mink} was proved by Minkowski \cite[\S 31, p.79]{Mink10};
a reference in English is \cite[Art.\ 45 p.149]{Hanc64}. We give two
proofs for the fact that only the standard lattice cube attains the upper
bound. First, in Section \ref{sec:GeometricProofOfMink} we use a
geometric argument due to Groemer \cite{Groe61}. Second, in Section
\ref{sec:MinkowskiStyleProofOfMink}, we give a self-contained proof.
The latter proof is based on congruences modulo $3$, in the line of
Minkowski's original approach.

\begin{ack}
   We thank Martin Henk for giving reference to \cite{Groe61}, and Josef
   Schicho for telling us about the $2$-dimensional case of
   Theorem \ref{main}.
\end{ack}

\section{Proof of Theorem \ref{main}}

Let a non-empty set $S \subseteq \R^{d}$ be given.  We define 
$\D(S)$ to be the set of vectors $v \in (\R^{d})^{*}$ such that $\sup
v(S) < \infty$ and $\inf v(S) > - \infty$. The set $\D(S)$ is easily
seen to be a linear subspace of $\R^{d}$. For $v \in \D(S)$, we define
the \emph{width of $S$ in direction $v$} to be
\begin{equation*}
   \width(S, v) := \sup v(S) - \inf v(S).
\end{equation*}
Note that the width does not change if we replace $S$ by the closure of
its convex hull. For $S$ convex the map sending $v$ to the first term
is sometimes referred to as the {\em St\"utzfunktion} of $S$. Now let
$\L(S)$ be the intersection of $\D(S)$ with the lattice $(\Z^{d})^*$. The
\emph{lattice width} of $S$ is given by
\begin{equation*}
   \width(S) = \inf \{ \width(S,v) \st v \in \L(S) \setminus \{0\} \}.
\end{equation*}
The set over which the infimum is taken may be empty, in which case
we set the lattice width equal to $\infty$. 
The set of {\em lattice width directions} of $S$ is defined as 
\begin{equation*}
   S' := \{ v \in \L(S)\setminus\{0\} \st \width(S, v) = \width(S) \}.
\end{equation*}

We now show how $\width(S)$ and $S'$ can be determined from a certain
compact convex set related to $S$. Let $e$ be the rank of $\L(S)$. Denote
by $\L(S)_\R$ the vector subspace of $(\R^d)^*$ spanned by $\L(S)$;
this is an $e$-dimensional, potentially strict, subspace of $\D(S)$. Let
$V \subseteq \R^d$ be the subspace of $\R^d$ where all elements of
$\L(S)_\R$ are zero, and let $\pi$ denote the natural projection $\R^d
\to \R^d/V$. This $\pi$ maps $\Z^d$ to a lattice $\Lambda$ of full
rank $e$ in the $e$-dimensional space $\R^d/V$, and the lattice dual to
this lattice is canonically isomorphic to $\L(S) \subseteq \L(S)_\R$.
The following lemma is straightforward, and reduces the study of
lattice width and lattice width directions to the case where $S$ has
$\D(S)=\R^d$, i.e., to {\em bounded} sets $S$.

\begin{lem}
The lattice width of $S$ relative to $\Z^d$ is equal to that of
$\pi(S)$ relative to $\Lambda$. Similarly, $S'$ equals $\pi(S)'$ under
the identification $\Lambda^*=\L(S)$.
\label{projection}
\end{lem}

Furthermore, if $S$ is bounded, then we can make it compact and convex
by passing to the closure of its convex hull. 

\begin{example}{\rm Let $S$ be given as $\R_{\geq 0}\,(1,\sqrt{2},0) + [0,1]\, (0,0,1) 
\subseteq \R^3$. Then $S$ is convex, unbounded, and contained in an affine hyperplane. 
Identify $\Z^3$ with $(\Z^3)^*$ via the usual scalar product. Then $\D(S) = \R (-\sqrt{2},1,0) + 
\R (0,0,1) \supsetneq \L(S) = \Z (0,0,1)$. In particular $\width(S) = 1$ and 
$S' = \pm (0,0,1)$. In the notation of Lemma \ref{projection}, we have $\pi(S) = [0,1] 
\subseteq \Lambda \cong \Z$.}
\end{example}

Note that $\conv S$ is not compact in the previous example.

\begin{prop}
Let $S \subseteq \R^d$ be a non-empty, compact and convex subset. Then $\width(S) < \infty$ if and only if $d > 0$. In this case:
\begin{enumerate}
\item If $\dim(S) < d$, then $\width(S) = 0$. Moreover, $S' \not=\emptyset$ if and only if $S$ is contained in an affine 
hyperplane with a rational defining vector.
\item If $\dim(S) = d$, then $\width(S) > 0$ and $S' \not=\emptyset$.
\end{enumerate}
\label{width-prop}
\end{prop}

\begin{proof}
(1) After translating $S$, which does not effect $\width(S)$ or $S'$, we
may assume that $S$ lies in the hyperplane through the origin defined by
a non-zero element $w \in (\R^d)^*$. If $w$ can be chosen in the lattice, 
then $\width(S) \leq \width(S,w)=0$ and $w \in S'$ and we are done. If
not, then the following argument shows that $\width(S)=0$ still holds,
while $S'=\emptyset$. Fix $\epsilon>0$ and consider the set
\[ Z:=\{v \in (\R^d)^* \mid v(S) \subseteq (-\epsilon/2,+\epsilon/2)\}. \]
By compactness of $S$ this set contains a $d$-dimensional ball $B$
centered at the origin. Moreover, $Z$ is stable under translation
over multiples of $w$. These facts imply that $Z$ has infinite
volume. Moreover, $Z$ is centrally-symmetric and convex since the interval
$(-\epsilon/2,+\epsilon/2)$ is. By Minkowski's well-known lattice point
theorem \cite{Mink10, Hanc64} $Z$ contains a non-zero lattice point
$v$. But then $\width(S,v)<\epsilon$.

(2) Since $S$ is compact, $\D(S) = \R^d$ and $\width(S) < \infty$. 
Since $S$ contains a ball $B$ of dimension $d$, it 
is clear that $\width(S,v)\geq \width(B,v)\geq \width(S)+1$ for $v$ 
outside some large ball in $(\R^d)^*$. This large ball has 
only finitely many lattice points, hence, $\width(S)$ is attained by one of 
these lattice points. In particular, $S' \not= \emptyset$. Since $S$ is not 
contained in an affine hyperplane, we have $\width(S) > 0$.
\end{proof}

\begin{example}{\rm Let us illustrate the previous proposition for $S = \{(0,0),(1,\sqrt{2})\} \subset \R^2$. 
Then there exist $a,b \in \Z\setminus\{0\}$ such that 
$\sqrt{2} \approx \frac{a}{b}$. Therefore, for $v := (a,-b)$ we see $\width(S,v) \approx 0$. 
Hence, $\width(S) = 0$. However, $S'=\emptyset$, since $\L(S) \cap S^\perp = \{0\}$. 
Moreover, note that for $S = \R_{\geq 0} (1,\sqrt{2})$ we have $\width(S) = \infty$.
}
\end{example}

Combining Lemma~\ref{projection} and Proposition~\ref{width-prop} yields the following 
observation.

\begin{cor}
Let $\emptyset \not= S \subseteq \R^d$ with $\dim(S) = d$. Then $\width(S) > 0$.
\label{observ}
\end{cor}

When $S$ is a full-dimensional and compact convex set, observe that
lattice width directions are necessarily primitive lattice vectors---that
is, they are not properly divisible by an integer.  The following result
shows that even more is true.

\begin{thm}
\label{thm:DualofS}
   Let $S$ be a subset of $\R^{d}$ such that $0 < \width(S) < \infty$. Then $\conv S'$ is a non-empty, 
   convex, centrally-symmetric set that contains no lattice point other than the origin in its relative
   interior. Moreover, the lattice points on the boundary of $\conv S'$
   are precisely the elements of $S'$.
\end{thm}

\begin{proof}
   Convexity and central-symmetry are immediate from the definition
   of $S'$. Non-emptiness follows from Lemma~\ref{projection} and Proposition~\ref{width-prop}.

   It is easy to verify that $\width(S, -) \maps \D(S) \to \R$ is
   a convex homogeneous function of degree $1$.  Suppose that $v \in
   \conv S'$.  By Carath\'eodory's theorem, there exist $v_{1}, \dotsc,
   v_{n} \in S'$ 
   and coefficients $0 \le \lambda_{1},
   \dotsc, \lambda_{n}$ with $\lambda_{1} + \dotsb + \lambda_{n} =
   1$ such that $v = \lambda_{1} v_{1} + \dotsb + \lambda_{n} v_{n}$.
   Hence, \begin{equation*}
      \width(S, v) \le \lambda_{1} \width(S, v_{1}) + \dotsb + \lambda_{n}
      \width(S, v_{n}) = \width(S).
   \end{equation*} In particular, if $v$ is a nonzero lattice point,
   so that $\width(S, v) \ge \width(S)$, then we have $\width(S,v) =
   \width(S)$, so that $v \in S'$.  Therefore, the non-zero lattice
   points in $\conv S'$ are precisely the elements of $S'$.

   Moreover, we cannot have $v$ in the relative interior of $\conv S'$,
   since, as shown above, that would imply that, for some $\epsilon > 0$,
   $\width(S, (1+\epsilon)v) \le \width(S)$, contradicting the fact that
      $\width(S, (1 + \epsilon)v) = (1 + \epsilon) \width(S, v) >
      \width(S)$ by minimality.
\end{proof}

\begin{remark}{\rm As we have seen, the notion of lattice width directions 
is quite subtle. Here is an important case where everything works out nice. Let $S$ be a {\em rational polyhedron}, 
namely, a convex set given by finitely many linear inequalities
\[S = \{x \in \R^d \;:\; f_i(x) \geq c_i \;\forall\, i = 1, \ldots, m\},\]
where $f_i \in (\Z^d)^*$ and $c_i \in \Z$. By standard arguments in convex geometry 
it follows that $\D(S) = \Rec(S)^\perp$, where 
\[\Rec(S) = \{y \in \R^d \;:\; \exists\, x \in S \text{ with } x + \R_{>0} y \subseteq S\}\]
is the {\em recession cone} of $S$. Since
\[\Rec(S) = \{y \in \R^d \;:\; f_i(y) \geq 0 \;\forall\, i = 1, \ldots, m\},\]
we see that $\D(S) = \L(S)_\R$ is the largest subspace contained in the rational polyhedral cone spanned by $f_1, \ldots, f_m$. 
The criterion in Theorem~\ref{thm:DualofS}, 
$0 < \width(S) < \infty$, holds if and only if $\dim(S) = d$ and $\dim(\Rec(S)) < d$.}
\end{remark}

We now show how Theorems \ref{mink} and \ref{thm:DualofS} imply
Theorem \ref{main}. Note that in general a full-dimensional compact convex set $S$ 
is not uniquely determined by $S'$, as exemplified in Figure~\ref{example}.

\begin{proof}[Proof of Theorem \ref{main}]
   We may assume $\width(S) < \infty$. By Corollary~\ref{observ} and 
   Theorem \ref{thm:DualofS}, we can apply Theorem \ref{mink} to
   $\conv S'$. This yields the desired upper bound $\card{S'} \leq 3^d-1$
   on the set $S'$ of lattice width directions.  Note that the bound is actually at
   most $3^{d-1}-1$ if $\conv S'$ does not have full dimension.  Hence,
   if $\card{S'}$ equals $3^d-1$, then $\conv S'$ is $d$-dimensional and
   by Theorem \ref{mink} there exists a lattice basis $e_1^*,\ldots,e_d^*$
   of $(\Z^d)^*$ such that $\conv S'$ is the standard lattice cube with vertices
   $\pm e_1 \pm \ldots \pm e_d$. After replacing $S$ by the closure of
   its convex hull we may assume that $S$ is closed and convex. Since
   all coordinates $e_i^*$ are bounded on $S$, the latter set is bounded,
   hence compact.

   We now show that $S$ is then a regular cross-polytope.  After
   translating $S$, we may assume that all coordinates take the
   same maximum $\lambda$ and the same minimum $-\lambda$ on $S$.
   For $i=1,\ldots,d$, let $p_i=\sum_{j=1}^d p_{ij} e_j \in S$ be a point
   with $i$-th coordinate $p_{ii}=\lambda$, and let $q_i=\sum_{j=1}^d
   q_{ij} e_j \in S$ be a point with $i$-th coordinate $q_{ii}=-\lambda$.
   By assumption, for every direction $v \in \{-1,0,1\}^d$, there exists
   a $t_v \in \R$ such that
   \begin{equation} \label{eq:*}\tag{$*$}
      -\lambda+t_v \leq v(p) \leq \lambda+t_v
   \end{equation} for all $p \in S$.  In particular, for distinct $i,j$
   we have \begin{align*}
      p_{ij} &\leq t_{e_i^*+e_j^*}, &  
      t_{e_i^*+e_j^*} &\leq q_{ij} \\
      -p_{ij} &\leq t_{e_i^*-e_j^*}, \text{ and } & 
      t_{e_i^*-e_j^*} &\leq -q_{ij},
   \end{align*} 
   so that $p_{ij}=t_{e_i^*+e_j^*}=-t_{e_i^*-e_j^*}=q_{ij}$.  Similarly,
   \begin{align*}
      p_{ji} &\leq t_{e_i^*+e_j^*}, 
      & t_{e_i^*+e_j^*} &\leq q_{ji} \\ 
      t_{e_i^*-e_j^*} &\leq p_{ji}, \text{ and } & 
      q_{ji} & \leq t_{e_i^*-e_j^*},
   \end{align*} so that $p_{ji}=t_{e_i^*+e_j^*}=t_{e_i^*-e_j^*}=q_{ji}$.
   Combining these, we find that $p_{ij}=q_{ij}=0$ for all distinct $i,j$,
   so that $p_i=\lambda e_i=-q_i$.  But then the inequalities \eqref{eq:*}
   for $v$, by filling in $p_i,q_i$ for some $i$ for which $v_i \neq 0$,
   give $t_v=0$ for all $v$.  The inequalities thus reduce to inequalities
   cutting out the cross-polytope spanned by $p_i$ and $q_i$.  Hence, $S$
   contains this cross-polytope and is contained in it.
\end{proof}

\section{A geometric proof of Theorem \ref{mink}}
\label{sec:GeometricProofOfMink}

In this section we give a geometric proof of Theorem \ref{mink}, inspired
by Minkowski's proof of his lattice point theorem.  It is based on
Groemer's article \cite{Groe61}. We start with the following
observation, a folklore result for which we could not find an explicit reference
in the literature.

\begin{thm}
Let $K \subseteq \R^d$ be a centrally-symmetric
convex set with $\Kiz=\{0\}$.  Then the union of the elements of
\begin{equation*}
      \mathscr{K} = \{ K + 2\alpha : \alpha \in \Kz \}
\end{equation*}
is contained in $3K$ and the relative interiors of these elements are pairwise disjoint.
\label{packing}
\end{thm}

\begin{proof}
For $x \in K$ and $\alpha \in \Kz$ we have
$x+2\alpha=3(\frac{1}{3}x+\frac{2}{3}\alpha)\in 3K$ by convexity of
$K$. This shows that $\bigcup \mathscr{K} \subseteq 3K$.
To see that the relative interiors of the elements of $\mathscr{K}$
are disjoint, suppose otherwise. Then there exist $x,y$ in
the relative interior of $K$ and distinct $\alpha,\beta \in \Kz$ such that
$x+2\alpha=y+2\beta$. By central symmetry $(x-y)/2$ is in the relative interior
of $K$, while it equals $\beta-\alpha$, which is a non-zero lattice
point. This contradicts the assumption that $\Kiz=\{0\}$.
\end{proof}

As a straightforward consequence of this result we can prove the
$3^d$-bound.

\begin{proof}[Proof of upper bound in Theorem \ref{mink}]
   Let $d' \leq d$ be the dimension of $K$.
   It follows from the theorem just proved that
   \begin{equation}
    \card{\Kz} \Vol(K) \le \Vol(3K) = 3^{d'} \Vol(K),
   \label{eq}
   \end{equation}
 so that $\card{\Kz} \leq 3^{d'} \leq 3^d$, as claimed.
\end{proof}

For the equality case we use Hilfssatz 2 of \cite{Groe61}.  For this
recall that a {\em parallelepiped} is any $\R^d$-translate of the convex
hull of the points $\pm e_1 \pm \cdots \pm e_d$ for an $\R$-basis $e_1,
\ldots, e_d$ of $\R^d$. By a {\em homothetic copy} of a subset $K$ of
$\R^d$ we mean any set of the form $\alpha+\lambda K$ for some $\alpha
\in \R^d$ and $\lambda \in \R_+$.

\begin{thm}[Groemer 1961]
Let $K \subseteq \R^d$ be a $d$-dimensional compact convex subset of
$\R^d$ that can be covered with finitely many homothetic copies of $K$,
whose interiors are mutually disjoint.  Then $K$ is a parallelepiped.
\label{groemer}
\end{thm}

Using this geometric result we can finish the proof of Theorem
\ref{mink}.

\begin{proof}[Proof of equality case in Theorem \ref{mink}]
   By Equation (\ref{eq}) we may assume that $K$ is $d$-dimensional. 
   Let us first argue that it suffices to consider the case where
   $K$ is compact.
   We already know by Minkowski's fundamental lattice point theorem
   that the volume of $K$ is bounded.  Now, let $\clK$ be the
   closure of $K$.  Then $\clK$ is also a $d$-dimensional
   centrally-symmetric convex set such that $\clKiz=\{0\}$.
   Therefore, the $3^d$-bound yields $\card{\clKz} = 3^d$.  Assume
   we already showed that $\clK$ is a standard lattice cube.  Since $K$
   has the same number of lattice points as $\clK$, all of the
   $2^d$ vertices of the standard lattice cube $\clK$ also have to be
   contained in $K$.  This shows $K=\clK$.

   Hence, we may assume that $K$ is compact. Now, by Theorem \ref{packing}
   we see that that the translates of $K$ by its $3^d$ lattice points
   together cover $3K$ and that their interiors do not intersect.
   Applying Theorem \ref{groemer} to $3K$ yields that $3K$ is a
   parallelepiped, hence so is $K$. By central symmetry, $K$ equals the
   convex hull of the $2^d$ vertices $\pm e_1 \pm \ldots \pm e_d$ for some
   $\R$-basis $e_1, \ldots, e_d$ of $\R^d$. It remains to show that $e_1,
   \ldots, e_d$ is a $\Z$-basis of $\Z^d$. We do this by arguing that
   there is only one way to cover the parallelepiped $3K$ with $3^d$
   translates of $K$, namely with the translates over the vectors $2
   \sum_{i=1}^d \epsilon_i e_i$ with each $\epsilon_i \in \{-1,0,1\}$.
   Indeed, this follows from a simple induction on $d$: consider any
   covering of $3K$ with $3^k$ translates of $K$. Then their interiors
   do not intersect for volume reasons. Now consider the facet $F$
   of $3K$ where the $e_1$-coordinate equals $3$. This facet is a
   $(d-1)$-dimensional parallelepiped which is covered by facets of
   translates $K_i$ of $K$. Since the interiors of the $K_i$ do not
   intersect, the relative interiors of their facets $F_i$ on $F$ do not
   intersect either. Hence, for volume reasons there are exactly $3^{d-1}$
   facets of translates of $K$ covering $F$. By induction, the $K_i$
   are the translates of $K$ over the vectors $2e_1 + 2 \sum_{i=2}^d
   \epsilon_i e_i$ with $\epsilon_i \in \{-1,0,1\}$. The same argument
   applies to the remaining two layers of $3K$ in the $e_1$-direction,
   so that the covering of $3K$ equals the standard covering above. Now,
   since the translates of $K$ over the vectors in $2\Kz$ also cover $3K$,
   the vectors $\sum_{i=1}^d \epsilon_i e_i$ with each $\epsilon_i \in
   \{-1,0,1\}$ are precisely the lattice points in $K$. In particular,
   all $e_i$ are in $\Z^d$.  Finally, they must generate
   $\Z^d$, or else $K$
   would contain more than $3^d$ lattice points. This proves that $K$
   is a standard cube.
\end{proof}

\section{A Minkowski-style proof of Theorem \ref{mink}}
\label{sec:MinkowskiStyleProofOfMink}

Minkowski's original proof of the $3^d$ bound relies on
considering congruences of lattice points.  By the same method he
also provided a sharpening of this bound in an important subcase.
Let us recall his elegant proof of these results.  For this let us
denote for a subset $S$ of $\R^d$ by $\ci(S)$ the relative
interior of $\conv(S)$.  Moreover, by $\Kbz$ we denote the set of
lattice points on the boundary of a convex set $K$.

\begin{thm}[Minkowski 1910]
   Let $K \subseteq \R^d$ be a $d$-dimensional centrally-symmetric
   convex set.
   \begin{enumerate}
      \item 
      If $\Kiz=\{0\}$, then $\card{\Kz} \leq 3^d$.  
      
      \item
      If $\Kiz=\{0\}$ and no boundary lattice point of $K$ is in
      the convex hull of some others, then $\card{\Kz} \leq
      2^{d+1}-1$.
   \end{enumerate}
   \label{old}
\end{thm}

\begin{proof}
   (1) We regard the canonical map $\gamma \maps \Z^d \to
   (\Z/3\Z)^d$.  We claim that $\gamma$ is injective on $\Kz$.
   For let $x,y \in \Kz$ such that $\gamma(x) = \gamma(y)$ be
   given.  Then $\gamma(x-y)=0$, so $z \deftobe (x-y)/3 \in \Z^d$.
   Since $K$ is centrally-symmetric, we see that $z \in
   \ci(0,x,-y) \subset \Kin$.  This implies $z \in \Kiz = \{0\}$,
   so $z=0$.  We deduce $x=y$, so $\gamma$ is injective, as
   claimed.

   (2) In this case, we look at the canonical map $\delta \maps
   \Z^d \to (\Z/2\Z)^d$.  Assume there is a boundary lattice point
   $v \in \Kbz$ such that $\delta(v) = 0$.  Then $v/2 \in \Z^d$,
   in particular, $0 \not= v/2 \in \Kiz$, a contradiction.  Hence,
   $\delta^{-1}(0) \cap \Kbz = \emptyset$.  Let $0 \not= f \in
   (\Z/2\Z)^d$ be fixed.  We claim that $\card{\delta^{-1}(f) \cap
   \Kbz} \leq 2$.  From this we immediately get the upper bound.
   So, assume that there are $x,y \in \Kbz$, $x \not= y \not= -x$,
   such that $\delta(x)=\delta(y)$.  Then $\delta(x-y)=0$ and
   therefore $z := (x-y)/2$ lies in $\Z^d$.  Since $K$ is
   centrally-symmetric and $x \not=-y$, we see that $z \in
   \ci(x,-y)$.  Since $x \not= y$, we have $z \not= 0$.  Therefore
   $z \in \Kbz$, a contradiction to the assumption.
\end{proof}

In the remainder of this section we prove Theorem \ref{mink}
following Minkowski's approach.  As it will turn out, it is enough to consider the case
of lattice polytopes.  For this let us recall that a {\em lattice
polytope} is the convex hull of finitely many lattice points in
$\Z^d$.  Now, the main idea is to use the modulo map to
inductively construct lattice points until we find a lattice point
in the interior of a facet.  This goal is inspired by the proof of
Theorem \ref{main} in a special case, see \cite[Theorem
6.1]{Nill06a}.  Then we show that $P$ has to be a prism over this
facet by applying a lattice point addition method analogous to
\cite[Lemma 5.9]{Nill06a}.  This allows us to proceed by induction
on the dimension.

\medskip {\em From now on let $d \geq 2$, and $P \subseteq \R^d$
be a $d$-dimensional centrally-symmetric lattice polytope with
$\Piz=\{0\}$ and $\card{\Pz} = 3^d$.} \medskip

The following result is the key-lemma for our proof.

\begin{lem}
\label{crux}
   For $x,y \in \Pz$ there exists a unique $z \in \Pz$ such that
   \begin{equation*}
      w \deftobe \frac{x+y+z}{3} \in \Z^d.
   \end{equation*}
   The lattice point lies in $\Pz$, and if $x \neq y$ then also $x
   \neq z \neq y$.
\end{lem}

\begin{proof}
   Consider the canonical map $\gamma \maps \Z^d \to (\Z/3\Z)^d$.
   As was shown in the proof of Theorem \ref{old}(1) the map
   $\gamma$ is injective on $\Pz$.  Since $\card{Pz}=3^d$ it is
   actually a bijection.  Therefore, there exists a unique $z \in
   \Pz$ such that $\gamma(x)+\gamma(y)+\gamma(z)=0$.  This latter
   equality is equivalent to $w \in \Z^d$.  The point $w$ is a
   convex combination of $x,y,z$ and hence lies in $\Pz$.
   Finally, if $x \neq y$ then $\gamma(x) \neq \gamma(y)$ and
   hence $\gamma(z)=-\gamma(x)-\gamma(y)$ equals neither
   $\gamma(x)$ nor $\gamma(y)$.  Hence $x \neq z \neq y$, as
   desired.
\end{proof}

We are going to use this observation in an inductive way.  For
this, let us write $F \leq P$ if $F$ is a face of $P$, and let us
denote by $\V(P)$ the set of vertices, \emph{i.e.},
$0$-dimensional faces, of $P$, and by $\F(P)$ the set of facets,
i.e., $(d-1)$-dimensional faces.  If $F$ is a facet of $P$, we
denote by $u_F \in (\Q^d)^*$ the unique outer normal of $F$
determined by $u_F(F) = 1$ and $u_F(P) \leq 1$.

\begin{prop}
   For $k = 1, \ldots, d-1$ there exists a face $F \lneq P$ such
   that $\dim(F) \geq k$ and $\Fiz \not= \emptyset$.
\end{prop}

\begin{proof}
   Let $k=1$ and assume the statement were false.  In this case,
   $\V(P) = \Pbz$, hence Theorem \ref{old}(2) yields $3^d =
   \card{\Pz} \leq 2^{d+1}-1$, in contradiction to $d \geq 2$.

   We proceed by induction.  Let $2 \leq k \leq d-1$.  Then, by
   the induction hypothesis, there exists a face $F \lneq P$ such
   that $\dim(F) \geq k-1$ and $\Fiz \not= \emptyset$.  We may
   assume that $\dim(F) = k-1$ and $x \in \Fiz$.  Let us choose a
   face $G$ of $P$ of dimension $k$ such that $F \subset G$.
   Since $k < d$, we have $G \not= P$.  Because $F$ is a facet of
   $G$, there exists a vertex $y \in \V(G)$, $y \notin F$, such
   that $\ci(x,y) \subseteq \Gin$.  Let $z,w$ be chosen as in
   Lemma \ref{crux}.  We distinguish two cases.

   \begin{enumerate}
      \item
      $\dim(x,y,z) = 1$.  We have three subcases to consider:
      \begin{enumerate}
         \item
         $x \in \ci(y,z)$.  Since $x \in F$, we get $y \in F$, a
         contradiction.
         
         \item
         $y \in \ci(x,z)$.  This is a contradiction to $y \in
         \V(P)$.
         
         \item
         $z \in \ci(x,y)$.  Hence, $z \in \Giz$, so $G$ satisfies
         the conditions of the Proposition, as desired.
      \end{enumerate}
      
      \item
      $\dim(x,y,z) = 2$.  Therefore, $w = (x+y+z)/3 \in
      \ci(x,y,z)$.  Let $H$ be a face of $P$ such that $w \in
      \Hiz$.  Then, $x,y,z \in H$.  In particular, $\ci(x,y)
      \subseteq H$, so also $G \subseteq H$.  Hence, $\dim(H) \geq
      \dim(G) = k$.  We claim that $H$ satisfies the conditions of
      the Proposition.  It remains to show that $H \not= P$.  So,
      assume $H = P$.  In this case, $w \in \Piz$, so $w = 0$, in
      particular, $x+y+z=0$.  Now, let $u$ be the unique outer
      normal of a facet of $P$ containing $G$.  By
      central-symmetry, $-u$ is also an outer normal of a facet of
      $P$.  However, $-u(z) = -u(-x-y) = 2$, in contradiction to
      $z \in P$.
   \end{enumerate}
\end{proof}

Applying the Proposition for $k=d-1$ yields:

\begin{cor}
   There exists a facet $F \in \F(P)$ such that $\Fiz \ne
   \emptyset$.
\label{facet}
\end{cor}

From now on, we will intensively use this corollary.

\begin{prop}
   Let $x \in \Fiz$ for $F \in \F(P)$.  Then \[x + (\Pz \setminus
   \Fz) \subseteq \Pz.\]
\label{add}
\end{prop}

\begin{proof}
   Let $y \in \Pz \setminus \Fz$.  Therefore, $\ci(x,y) \subseteq
   \Pin$.  We may assume $y \notin \{0, -x\}$.  Let $z,w$ be
   chosen as in Lemma \ref{crux}.  Again, we distinguish two
   cases.

   \begin{enumerate}
      \item
      $\dim(x,y,z) = 1$. 
      \begin{enumerate}
         \item
         $x \in \ci(y,z)$.  Since $x \in F$, we get $y \in F$, a
         contradiction.
         
         \item
         $y \in \ci(x,z)$.  Since $y \notin F$, we get $z \notin
         F$.  Therefore, $y \in \Piz = \{0\}$, a contradiction.
         
         \item
         $z \in \ci(x,y)$.  Hence, $z \in \Piz = \{0\}$, so
         necessarily $y=-x$, a contradiction.
      \end{enumerate}
      
      \item
      $\dim(x,y,z) = 2$.  Therefore, $w = (x+y+z)/3 \in
      \ci(x,y,z)$.  Hence, $w \in \Piz = \{0\}$.  This implies
      $x+y = -z \in \Pz$ by central-symmetry, as desired.
   \end{enumerate}
\end{proof}

Here is a direct consequence.  For this let us define $u_F^\perp
\deftobe \{v \in \R^d \st u_F(v) = 0\}$ for a facet $F \in
\F(P)$.

\begin{cor}
   Let $x \in \Fiz$ for $F \in \F(P)$.  Then
   \begin{equation*}
      \Pz = \Fz \sqcup (\Pz \cap u_F^\perp) \sqcup (-\Fz).
   \end{equation*}
   Moreover, the map $\Z^d \to \Z^d$, $y \mapsto x+y$, induces
   bijections
   \begin{equation*}
      (-\Fz) \to (\Pz \cap u_F^\perp) \to \Fz.
   \end{equation*}
   In particular, $\V(P) \subseteq \Fz \sqcup (-\Fz)$.
\label{layer}
\end{cor}

\begin{proof}
   Assume the first statement is wrong.  Then, by
   central-symmetry, there exists $y \in \Pz$, $y \notin \Fz$,
   such that $u_F(y) > 0$.  Proposition \ref{add} yields $x+y \in
   \Pz$.  However, $u_F(x+y) > 1$, a contradiction.

   The second statement follows by central-symmetry from
   Proposition \ref{add}.  For the last statement, note that,
   since $P$ is a lattice polytope, we have $\V(P) \subseteq \Pz$.
   So assume $y \in \V(P)$ with $y \in u_F^\perp$.  Then $y \in
   \ci(-x+y,x+y)$ with $-x+y \in -\Fz$ and $x+y \in \Fz$, a
   contradiction.
\end{proof}

Now, we can easily finish the proof of Theorem \ref{mink}.

\begin{proof}[Proof of Theorem \ref{mink}]
   By Theorem \ref{old}(1) it remains to prove the equality case.
   For this, let us first deal with the case of a lattice polytope
   $P$ as before.  The proof is by induction on the dimension $d$.
   We may assume $d \geq 2$.  By Corollary \ref{facet} there
   exists a facet $F \in \F(P)$ such that $F$ has an interior
   lattice point $x$.  Now, Corollary \ref{layer} actually shows
   that $-F=F-2x$ and $P = \conv(F,F-2x)$, i.e., $P$ is a prism
   over $F$.  Moreover, we see that $F-x = P \cap u_F^\perp$ is a
   $(d-1)$-dimensional centrally-symmetric lattice polytope (with
   respect to the lattice $\Z^d \cap u_F^\perp$) such that
   $\intr{(F-x)}=\{0\}$ and $\card{(F-x)_\Z} = 3^{d-1}$.  Hence,
   the induction hypothesis yields that $F-x$ is a standard lattice cube
   (with respect to a lattice basis $e_1, \ldots, e_{d-1}$).  It
   remains to show that $e_1, \ldots, e_{d-1},x$ is a lattice
   basis of $\Z^d$.  This follows, since any lattice point in
   $\Z^d$ can be successively translated via $e_1, \ldots,
   e_{d-1}, x$ into $P$, and $\Pz \subseteq \{\pm e_1 + \cdots +
   \pm e_{d-1} \pm x\}$ by Corollary \ref{layer}.

   In the general case, let $K \subseteq \R^d$ be a
   $d$-dimensional centrally-symmetric convex set with
   $\Kiz=\{0\}$ and $\card{\Kz} = 3^d$.  We define $P \deftobe
   \conv(\Kz)$.  This is a centrally-symmetric lattice polytope
   with $\Piz=\{0\}$ and $\card{\Pz} = 3^d$, in particular
   $\dim(P) = d$ by Theorem \ref{old}(1).  Therefore, $P$ is a
   standard lattice cube (with respect to a lattice basis $e_1, \ldots,
   e_d$).  Assume $P \subsetneq K$.  Then there exists $x \in K$,
   $x \notin P$.  Hence, there is a facet $F \in \F(P)$ such that
   $u_F(x) > 1$.  We may assume $F = e_1 + [-1,1] e_2 + \cdots +
   [-1,1] e_d$.  However, this implies that $e_1 \in
   \ci(0,\V(F),x)$, a contradiction to $\Kiz = \{0\}$.
\end{proof}


\def\cprime{$'$}
\providecommand{\bysame}{\leavevmode\hbox to3em{\hrulefill}\thinspace}
\providecommand{\MR}{\relax\ifhmode\unskip\space\fi MR }
\providecommand{\MRhref}[2]{%
  \href{http://www.ams.org/mathscinet-getitem?mr=#1}{#2}
}
\providecommand{\href}[2]{#2}

\end{document}